\documentclass[11pt]{article}

\usepackage[T1]{fontenc}
\usepackage{amsmath,amssymb,amsthm}
\usepackage{graphicx}
\usepackage{tikz}
\usepackage[a4paper,margin=1in]{geometry}
\usepackage[numbers,sort&compress]{natbib}
\usepackage[colorlinks=true,allcolors=blue]{hyperref}

\newtheorem{theorem}{Theorem}
\newtheorem{lemma}[theorem]{Lemma}
\newtheorem{proposition}[theorem]{Proposition}
\newtheorem{corollary}[theorem]{Corollary}
\theoremstyle{definition}
\newtheorem{definition}{Definition}

\newcommand{\qd}{\operatorname{qd}}
\newcommand{\score}{\operatorname{sc}}

\title{The maximum quartet distance between phylogenetic trees}
\author{Lior Pachter\thanks{The author used GPT-5.6 to prove the theorem and
to draft an initial version of this manuscript.  The author assisted with the
proof strategy, verified the mathematical arguments, edited the manuscript,
and takes full responsibility for the content.}\\
Division of Biology and Biological Engineering\\
Department of Computing and Mathematical Sciences\\
California Institute of Technology\\
Correspondence should be addressed to
\href{mailto:lpachter@caltech.edu}{\texttt{lpachter@caltech.edu}}.}
\date{}

\begin{document}
\maketitle

\begin{abstract}
The quartet distance counts the four-leaf subsets on which two binary
phylogenetic trees display different topologies.  We prove that its maximum
over trees on $n$ leaves is
\[
 \left(\frac23+o(1)\right)\binom n4,
\]
resolving a conjecture of Bandelt and Dress from 1986 and calibrating the
scale of a fundamental metric for comparing phylogenetic trees.  The proof
reduces arbitrary pairs of trees to caterpillars by means of a common-root
planarization and an identity on five-leaf trees.
\end{abstract}

\section{Introduction}

The quartet distance was introduced by Estabrook, McMorris, and Meacham
as a measure of dissimilarity between phylogenetic trees
\cite{Estabrook_1985}.  Every four-element subset of the leaf set induces
one of three resolved quartet topologies, and the distance between two
binary trees is the number of subsets on which those topologies differ.
Beyond their use in comparing trees, quartets also underlie reconstruction
methods; in particular, neighbor joining admits a robust quartet
interpretation \cite{Mihaescu_2009}.
The extremal problem is therefore to determine the largest possible
distance between two trees on the same $n$ leaves.

Bandelt and Dress proved that this maximum is strictly smaller than
$\frac{14}{15}\binom n4$ for $n\,\geq\,6$ and conjectured that its asymptotic
value is $\frac23\binom n4$ \cite{Bandelt_1986}.  The constant $2/3$ is
forced by a random relabelling: for every fixed quartet in one tree, a
uniformly relabelled second tree displays each of its three topologies
with probability $1/3$.  Related distributional properties of the
quartet metric were subsequently studied by Steel and Penny
\cite{Steel_1993}.

Progress on the upper bound came much later.  Alon, Snir, and Yuster
proved an asymptotic upper bound of
$\left(\frac9{10}+o(1)\right)\binom n4$
\cite{Alon_2014}.  Alon, Naves, and Sudakov then used flag algebras
\cite{Razborov_2007} to improve this to
$\left(0.69+o(1)\right)\binom n4$ and proved the conjectured $2/3$ bound
when both trees are caterpillars \cite{Alon_2016}.  Chor, Erd\H{o}s, and Komornik
gave an explicit family of pairs of complete balanced trees whose
distance is asymptotic to the conjectured value \cite{Chor_2019}.
Snir, Weissberg, and Yuster studied a partial-information extension of
the extremal problem and showed that the analogue of the Bandelt--Dress
phenomenon need not hold when some taxa have prescribed locations
\cite{Snir_2021}.  More recently, Czabarka, Kelk, Moulton, and Sz\'ekely
realized quartet distance as one endpoint of a family of tree metrics
arising from coconvex characters, relating their diameters to extremal
character counts \cite{Czabarka_2026}.

We prove the conjecture by reducing the problem for arbitrary pairs of
trees to the caterpillar case.

\begin{definition}[phylogenetic trees and displayed quartets]
\label{def:phylogenetic-tree}
An \emph{unrooted binary phylogenetic tree} on a finite set $X$ is a tree
whose leaves are bijectively labelled by $X$ and whose internal vertices
have degree three.  For $Q\in\binom X4$, let $T|_Q$ be the quartet topology
obtained from the minimal subtree spanning $Q$ by suppressing vertices of
degree two.
\end{definition}

\begin{definition}[quartet distance and its maximum]
\label{def:quartet-distance}
For two binary phylogenetic trees on $X$, define
\[
 \qd(T_1,T_2)
 \,=\,\bigl|\{Q\in\tbinom X4:T_1|_Q\,\ne\,T_2|_Q\}\bigr|.
\]
For $|X|\,=\,n$, let $M_n$ be the maximum of this distance over all pairs
of trees on $X$.
\end{definition}

\begin{theorem}\label{thm:main}
The maximum quartet distance satisfies
\[
 M_n\,=\,\left(\frac23+o(1)\right)\binom n4.
\]
\end{theorem}

In other words, no two sufficiently large trees can disagree on
substantially more than two thirds of their quartets, while random
relabelling shows that this proportion is asymptotically attained.

The reduction has two ingredients.  First, a uniformly random planar
embedding of a tree, cut at a chosen leaf $r$, linearly orders all the leaves
with $r$ as an endpoint and thereby produces a random caterpillar $C_r(T)$.
Second, when both trees
are rooted at the same leaf, an exact identity on each five-leaf set shows
how the expected caterpillar scores recover the original quartet score.
The caterpillar bound then gives the general result.  The structure of the
proof and its correspondence with the Lean formalization (available
at \url{https://github.com/pachterlab/P_2026_Bandelt-Dress}) are illustrated
in Figure~\ref{fig:proof-structure}.

\begin{figure}[p]
 \centering
 \includegraphics[width=\textwidth,height=.82\textheight,keepaspectratio]
 {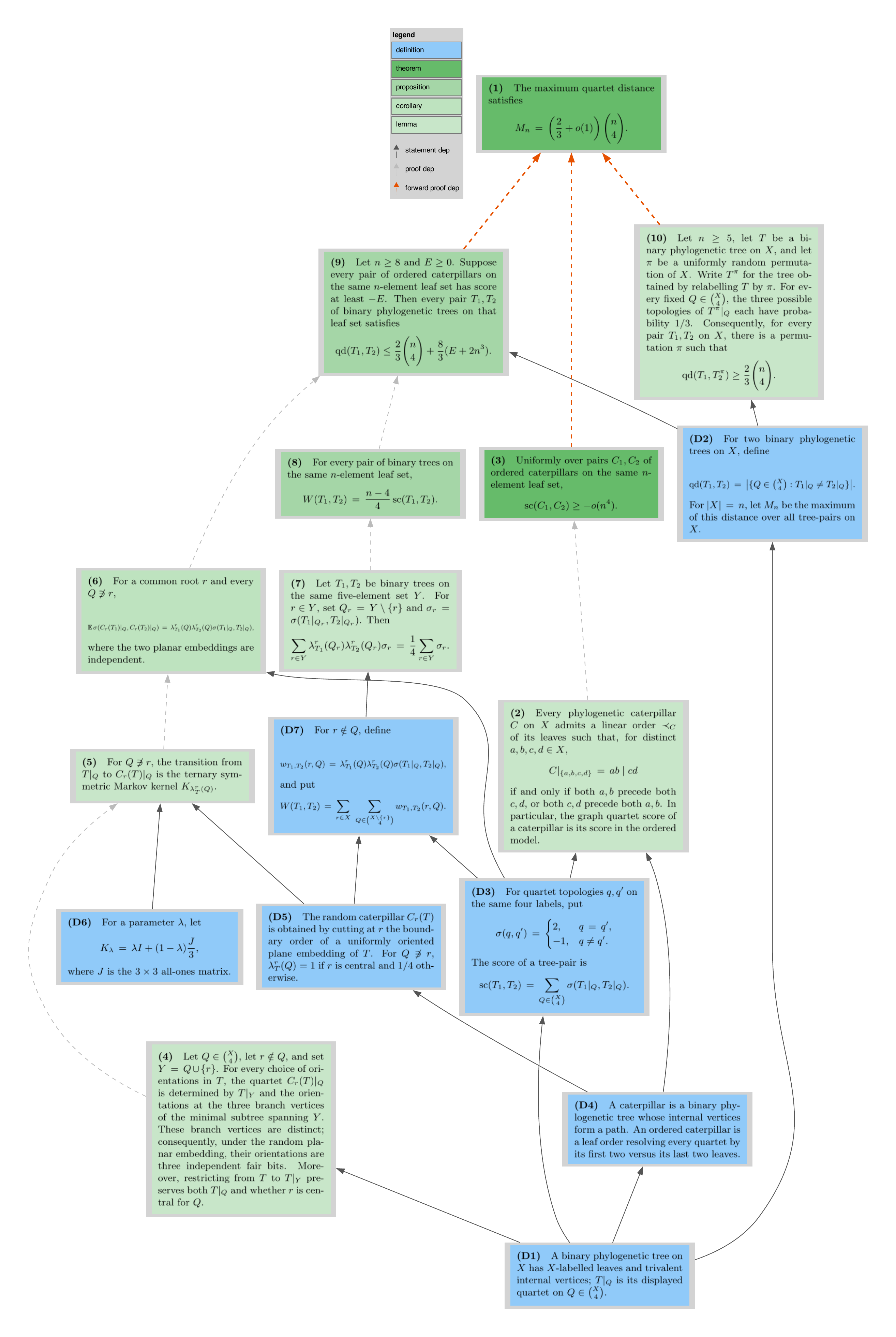}
 \caption{Proof structure and formal correspondence.  This diagram was
 generated with \texttt{span}, introduced in \cite{Pachter_2026}, which
 links labelled LaTeX statements to their corresponding Lean declarations
 and lifts Lean dependencies back to the paper.  Blue boxes are definitions
 and green boxes are results; solid arrows denote statement dependencies and
 dashed arrows denote proof dependencies.}
 \label{fig:proof-structure}
\end{figure}

\section{Quartet score and the caterpillar theorem}

\begin{definition}[quartet score]\label{def:quartet-score}
For quartet topologies $q,q'$ on the same four labels, let
\[
 \sigma(q,q')\,=\,
 \begin{cases}
  2,&q\,=\,q',\\
 -1,&q\,\ne\,q'.
 \end{cases}
\]
The score of a pair of trees is
\[
 \score(T_1,T_2)
 \,=\,\sum_{Q\in\binom X4}\sigma(T_1|_Q,T_2|_Q).
\]
\end{definition}

If $N\,=\,\binom n4$ and $D\,=\,\qd(T_1,T_2)$, then
\begin{equation}\label{eq:score-distance}
 \score(T_1,T_2)\,=\,2(N-D)-D\,=\,2N-3D.
\end{equation}
Consequently, the upper bound in Theorem~\ref{thm:main} is equivalent to
$\score(T_1,T_2)\,\geq\,-o(n^4)$ uniformly over all pairs of trees.

\begin{definition}[caterpillars]\label{def:caterpillar-model}
A binary phylogenetic tree is a \emph{caterpillar} if its internal
vertices induce a path, called its spine.  An \emph{ordered caterpillar}
on $X$ is a linear order of $X$ together with the following induced quartet
system: each quartet is resolved by putting its first two leaves against
its last two.  Its score against another ordered caterpillar is defined by
the formula in Definition~\ref{def:quartet-score}, using these induced
quartet topologies.  The following lemma connects graph-theoretic
caterpillars to this ordered model.
\end{definition}

\begin{lemma}[ordered caterpillar model]\label{lem:caterpillar-order}
Every phylogenetic caterpillar $C$ on $X$ admits a linear order $\prec_C$
of its leaves such that, for distinct $a,b,c,d\in X$,
\[
 C|_{\{a,b,c,d\}}\,=\,ab\mid cd
\]
if and only if both $a,b$ precede both $c,d$, or both $c,d$ precede both
$a,b$.  In particular, the graph quartet score of a caterpillar is its
score in the ordered model.  Conversely, if $|X|\,\geq\,4$, every linear
order on $X$ is realized in this way by a phylogenetic caterpillar.
\end{lemma}

\begin{proof}
Order the leaves by the positions of their incident vertices along the
internal spine, breaking the two ties at each endpoint arbitrarily.  The
middle spine edge of any four selected leaves separates precisely the
first two from the last two.  Summing the resulting quartet scores gives
the score assertion.  Conversely, given a linear order on a set $X$ with
$|X|\,\geq\,4$, take a path of $|X|-2$ internal vertices.  Attach the first
two leaves to one endpoint, the last two leaves to the other endpoint, and
the intervening leaves, in order, to the successive interior vertices.
The resulting phylogenetic caterpillar has the prescribed quartet system.
\end{proof}

We use the following theorem.

\begin{theorem}[Alon--Naves--Sudakov \cite{Alon_2016}]
\label{thm:caterpillar}
There are nonnegative numbers $\varepsilon_n$, tending to zero as
$n\,\to\,\infty$, such that every pair $C_1,C_2$ of ordered caterpillars
on the same $n$-element leaf set satisfies
\[
 \score(C_1,C_2)\,\geq\,-\varepsilon_n n^4.
\]
\end{theorem}

\begin{proof}[Justification of the formulation]
Theorem~1.2 of Alon, Naves, and Sudakov states that the maximum quartet
distance between two $n$-leaf phylogenetic caterpillars is at most
$\left(2/3+o(1)\right)\binom n4$.  Lemma~\ref{lem:caterpillar-order}
identifies their graph-theoretic caterpillars with the ordered model, and
\eqref{eq:score-distance} converts their distance bound into the uniform
score bound stated here.
\end{proof}

\section{The common-root planar channel}

\begin{definition}[common-root planarization]\label{def:planarization}
Fix a leaf $r$ of a binary tree $T$.  At every internal vertex, choose
independently and uniformly one of its two cyclic orientations.  The
resulting plane tree has a cyclic boundary order of its leaves.  Cut this
order at $r$, retaining $r$ as the first element.  The resulting linear
order of the full leaf set $X$ defines an $|X|$-leaf caterpillar, denoted
$C_r(T)$: if four leaves occur in the order $x_1,x_2,x_3,x_4$, their
displayed quartet is $x_1x_2\mid x_3x_4$.

Let $Q$ be a quartet not containing $r$.  Call $r$ \emph{central for
$Q$ in $T$} if $r$ is the unique leaf incident with the middle internal
vertex of the five-leaf tree $T|_{Q\cup\{r\}}$.  Define
\begin{equation}\label{eq:lambda}
 \lambda_T^r(Q)\,=\,
 \begin{cases}
  1,&r\text{ is central for }Q,\\
  \frac14,&r\text{ is not central for }Q.
 \end{cases}
\end{equation}
\end{definition}

\begin{definition}[ternary symmetric Markov kernel]\label{def:symmetric-kernel}
For $-\frac12\,\leq\,\lambda\,\leq\,1$, let
\[
 K_\lambda\,=\,\lambda I+(1-\lambda)\frac{J}{3},
\]
where $J$ is the $3\times3$ all-ones matrix.  After fixing an order of the
three quartet topologies, rows index the input topology and columns index
the output topology.  Equivalently,
\[
 K_\lambda
 \,=\,
 \begin{pmatrix}
  \frac{1+2\lambda}{3} & \frac{1-\lambda}{3} & \frac{1-\lambda}{3}\\
  \frac{1-\lambda}{3} & \frac{1+2\lambda}{3} & \frac{1-\lambda}{3}\\
  \frac{1-\lambda}{3} & \frac{1-\lambda}{3} & \frac{1+2\lambda}{3}
 \end{pmatrix}.
\]
In particular,
\[
 K_1\,=\,I,
 \qquad
 K_{1/4}
 \,=\,
 \begin{pmatrix}
  \frac12 & \frac14 & \frac14\\
  \frac14 & \frac12 & \frac14\\
  \frac14 & \frac14 & \frac12
 \end{pmatrix}.
\]
\end{definition}

\begin{lemma}[five-leaf factorization]\label{lem:five-leaf-factorization}
Let $Q\in\binom X4$, let $r\notin Q$, and set $Y\,=\,Q\cup\{r\}$.
For every choice of orientations in $T$, the quartet
$C_r(T)|_Q$ is determined by $T|_Y$ and the orientations at the three
branch vertices of the minimal subtree spanning $Y$.  These branch
vertices are distinct; consequently, under the random planar embedding,
their orientations are three independent fair bits.  Moreover,
restricting from $T$ to $T|_Y$ preserves both $T|_Q$ and whether $r$ is
central for $Q$.
\end{lemma}

\begin{proof}
The minimal subtree spanning $Y$ is a tree with the five selected leaves.
Suppressing its degree-two vertices leaves exactly three trivalent branch
vertices, each inherited from a distinct internal vertex of $T$.
Each pendant subtree removed in restricting to $Y$ occupies an interval
of the boundary order, and suppressing a degree-two vertex does not change
the cyclic order of the remaining leaves.  Thus restricting the ambient
boundary order to $Y$ gives the boundary order of the reduced five-leaf
plane tree.
Suppressing the paths between its branch vertices changes neither which
leaves lie on the two sides of an edge nor the cyclic order induced at a
branch vertex.  It follows that the displayed input quartet, the central
leaf, and the output boundary-order quartet are all read from this
five-leaf tree and its three inherited orientations.  Since the ambient
orientations were chosen independently, the three inherited bits are
independent and uniform.
\end{proof}

\begin{lemma}[planar channel]\label{lem:channel}
For $Q\not\ni r$, the transition from $T|_Q$ to $C_r(T)|_Q$ is the
ternary symmetric Markov kernel $K_{\lambda_T^r(Q)}$.
\end{lemma}

\begin{proof}
By Lemma~\ref{lem:five-leaf-factorization}, it suffices to analyze the
eight orientation choices of a five-leaf binary tree, which consists of
a central leaf and two cherries.

If $r$ is the central leaf, cutting the boundary order at $r$ leaves the
two cherries as consecutive blocks, so all eight choices of the three
orientations display $T|_Q$.  If $r$ is not central, the eight choices
may be labelled so that the cherries are $\{r,a\}$ and $\{b,c\}$ and the
central leaf is $d$.  Starting immediately after $r$, the resulting orders
on $Q$ are
\[
 adbc,\ adcb,\ dbca,\ dcba,\ abcd,\ acbd,\ bcda,\ cbda.
\]
The first, second, seventh, and eighth orders display
$ad\mid bc\,=\,T|_Q$;
the remaining four split evenly between $ab\mid cd$ and $ac\mid bd$.
Thus the eight choices display $T|_Q$ four times and each alternative
topology twice.  Therefore the
transition probabilities are $(1,0,0)$ in the first case and
$(1/2,1/4,1/4)$ in the second, which are precisely the rows of the stated
matrices for $\lambda\,=\,1$ and $\lambda\,=\,1/4$.
\end{proof}

\begin{corollary}\label{cor:channel-score}
For a common root $r$ and every $Q\not\ni r$,
\begin{equation}\label{eq:channel-score}
 \mathbb E\,\sigma(C_r(T_1)|_Q,C_r(T_2)|_Q)
 \,=\,\lambda_{T_1}^r(Q)\lambda_{T_2}^r(Q)
  \sigma(T_1|_Q,T_2|_Q),
\end{equation}
where the two planar embeddings are independent.
\end{corollary}

\begin{proof}
The score matrix on the three quartet states is $3I-J$.  Since
$J(3I-J)\,=\,(3I-J)J\,=\,0$, applying Lemma~\ref{lem:channel} independently in
the two trees gives the stated identity.
\end{proof}

\section{The five-leaf identity}

\begin{definition}[rooted and total weighted scores]\label{def:rooted-weight}
For $r\notin Q$, define
\[
 w_{T_1,T_2}(r,Q)\,=\,
 \lambda_{T_1}^r(Q)\lambda_{T_2}^r(Q)
 \sigma(T_1|_Q,T_2|_Q),
\]
and let
\[
 W(T_1,T_2)\,=\,
 \sum_{r\in X}\ \sum_{Q\in\binom{X\setminus\{r\}}4}
 w_{T_1,T_2}(r,Q).
\]
\end{definition}

\begin{lemma}\label{lem:five-leaf}
Let $T_1,T_2$ be binary trees on the same five-element set $Y$.  For
$r\in Y$, set $Q_r\,=\,Y\setminus\{r\}$ and
$\sigma_r\,=\,\sigma(T_1|_{Q_r},T_2|_{Q_r})$.  Then
\begin{equation}\label{eq:five-leaf}
 \sum_{r\in Y}
 \lambda_{T_1}^r(Q_r)\lambda_{T_2}^r(Q_r)\sigma_r
 \,=\,\frac14\sum_{r\in Y}\sigma_r.
\end{equation}
\end{lemma}

\begin{proof}
Let $a$ and $b$ be the central leaves of $T_1$ and $T_2$.  If $a\,=\,b$,
then either the two cherry matchings agree, in which case every
$\sigma_r\,=\,2$, or they differ, in which case every $\sigma_r\,=\,-1$.  Calling
this common value $t$, the left side of \eqref{eq:five-leaf} is
\[
 \left(1+4\cdot\frac1{16}\right)t\,=\,\frac54t,
\]
which equals one quarter of the sum of the five scores.

Suppose $a\,\ne\,b$, and call the other leaves $c,d,e$.  Relabel them so
that the cherries of $T_1$ are $\{b,c\}$ and $\{d,e\}$.  The three
possible cherry matchings in $T_2$ are
\[
 \{a,c\},\{d,e\},\qquad
 \{a,d\},\{c,e\},\qquad
 \{a,e\},\{c,d\}.
\]
For deletions $r\in\{c,d,e\}$, respectively, each matching gives exactly
one agreement and two disagreements.  Hence
$\sigma_c+\sigma_d+\sigma_e\,=\,2-1-1\,=\,0$.  The channel weight is $1/4$ at
$a$ and $b$, and $1/16$ at $c,d,e$.  The left side of
\eqref{eq:five-leaf} minus its right side is therefore
\[
 \left(\frac1{16}-\frac14\right)
 (\sigma_c+\sigma_d+\sigma_e)\,=\,0.
\]
\end{proof}

Figure~\ref{fig:five-leaf-channel} summarizes the two local channel cases
and the cancellation that results when the same root is used in both trees.

\begin{figure}[t]
 \centering
 \begin{tikzpicture}[
   x=1cm,
   y=1cm,
   internal/.style={circle,fill=black,inner sep=1.6pt},
   leaf/.style={circle,draw=black,fill=white,minimum size=5.5mm,
     inner sep=0pt,font=\small},
   root/.style={leaf,very thick},
   panel/.style={draw=black!45,rounded corners=2pt},
   note/.style={font=\small,align=center}
 ]
  \draw[panel] (-7.8,-2.55) rectangle (-0.2,2.45);
  \node[font=\small\bfseries] at (-4,2.15)
    {(a) The root $r$ is central};

  \node[internal] (cl) at (-6,0.75) {};
  \node[internal] (cm) at (-4,0.75) {};
  \node[internal] (cr) at (-2,0.75) {};
  \node[leaf] (ca) at (-7,1.45) {$a$};
  \node[leaf] (cb) at (-7,0.05) {$b$};
  \node[root] (croot) at (-4,1.65) {$r$};
  \node[leaf] (cc) at (-1,1.45) {$c$};
  \node[leaf] (cd) at (-1,0.05) {$d$};
  \draw (cl)--(cm)--(cr);
  \draw (cl)--(ca) (cl)--(cb) (cm)--(croot) (cr)--(cc) (cr)--(cd);

  \node[note] at (-4,-0.35)
    {$Q\,=\,\{a,b,c,d\}$,\quad $T|_Q\,=\,ab\mid cd$};
  \node[note] at (-4,-0.95)
    {all eight embeddings give $ab\mid cd$};
  \node[note] at (-4,-1.50)
    {output probabilities $(1,0,0)$};
  \node[note] at (-4,-2.05)
    {$\lambda_T^r(Q)\,=\,1$};

  \draw[panel] (0.2,-2.55) rectangle (7.8,2.45);
  \node[font=\small\bfseries] at (4,2.15)
    {(b) The root $r$ is not central};

  \node[internal] (nl) at (2,0.75) {};
  \node[internal] (nm) at (4,0.75) {};
  \node[internal] (nr) at (6,0.75) {};
  \node[root] (nroot) at (1,1.45) {$r$};
  \node[leaf] (na) at (1,0.05) {$a$};
  \node[leaf] (nd) at (4,1.65) {$d$};
  \node[leaf] (nb) at (7,1.45) {$b$};
  \node[leaf] (nc) at (7,0.05) {$c$};
  \draw (nl)--(nm)--(nr);
  \draw (nl)--(nroot) (nl)--(na) (nm)--(nd) (nr)--(nb) (nr)--(nc);

  \node[note] at (4,-0.35)
    {$Q\,=\,\{a,b,c,d\}$,\quad $T|_Q\,=\,ad\mid bc$};
  \node[note] at (4,-1.05)
    {the eight embeddings give\\[-1pt]
     $4(ad\mid bc),\ 2(ab\mid cd),\ 2(ac\mid bd)$};
  \node[note] at (4,-1.65)
    {output probabilities $(1/2,1/4,1/4)$};
  \node[note] at (4,-2.15)
    {$\lambda_T^r(Q)\,=\,1/4$};

  \node[draw=black!60,rounded corners=2pt,fill=black!3,
    minimum width=15.6cm,minimum height=1.55cm,align=center,font=\small]
    at (0,-3.65) {\textbf{Common-root five-leaf identity}\\[5pt]
    $\displaystyle
     \sum_{r\in Y}
     \lambda_{T_1}^r(Q_r)\lambda_{T_2}^r(Q_r)\sigma_r
     \,=\,\frac14\sum_{r\in Y}\sigma_r,
     \qquad Q_r\,=\,Y\setminus\{r\}.$};
 \end{tikzpicture}
 \caption{The five-leaf planar channel and its cancellation.  A binary
 five-leaf tree has three internal vertices, whose independent orientations
 give eight plane embeddings.  If the common root $r$ is central, all eight
 embeddings preserve the input quartet.  If it is not central, the input
 quartet occurs four times and each alternative occurs twice.  The boxed
 identity is Lemma~\ref{lem:five-leaf}; using the same root in both trees is
 what makes the channel weights cancel after summation.}
 \label{fig:five-leaf-channel}
\end{figure}
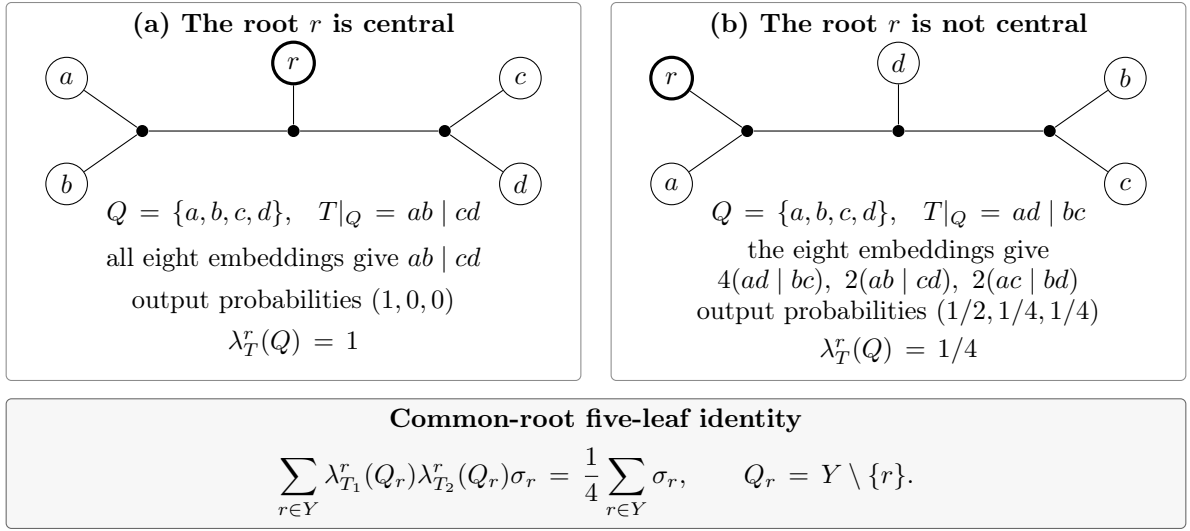

\begin{proposition}[unweighting identity]\label{prop:unweighting}
For every pair of binary trees on the same $n$-element leaf set,
\begin{equation}\label{eq:unweighting}
 W(T_1,T_2)\,=\,\frac{n-4}{4}\score(T_1,T_2).
\end{equation}
\end{proposition}

\begin{proof}
Apply Lemma~\ref{lem:five-leaf} to $T_1|_Y,T_2|_Y$ and sum over
$Y\in\binom X5$.  Centrality is determined by the restriction to
$Q\cup\{r\}\,=\,Y$.  Consequently, the left sides sum to $W(T_1,T_2)$:
each pair $(r,Q)$ with $r\notin Q$ occurs once, for
$Y\,=\,Q\cup\{r\}$.  On the right, each quartet occurs for all $n-4$ choices
of the fifth leaf.  This proves \eqref{eq:unweighting}.
\end{proof}

\section{Proof of the main theorem}

\begin{proposition}[quantitative common-root reduction]
\label{prop:common-root-reduction}
Let $n\,\geq\,8$ and $E\,\geq\,0$.  Suppose every pair of ordered caterpillars on the
same $n$-element leaf set has score at least $-E$.  Then every pair
$T_1,T_2$ of binary phylogenetic trees on that leaf set satisfies
\[
 \qd(T_1,T_2)
 \,\leq\,\frac23\binom n4+
 \frac{4n}{3(n-4)}\left(E+2\binom{n-1}{3}\right)
 \,\leq\,\frac23\binom n4+\frac83(E+2n^3).
\]
\end{proposition}

\begin{proof}
Fix $r\in X$.  Every realization of $C_r(T_1),C_r(T_2)$ is a pair of
ordered caterpillars, so its score is at least $-E$.  Only
$\binom{n-1}{3}$ quartets contain $r$, and every quartet score has
absolute value at most $2$.  Deleting those terms and using
Corollary~\ref{cor:channel-score} gives
\begin{equation}\label{eq:fixed-root}
 \sum_{Q\in\binom{X\setminus\{r\}}4}
 \lambda_{T_1}^r(Q)\lambda_{T_2}^r(Q)
 \sigma(T_1|_Q,T_2|_Q)
 \,\geq\,-E-2\binom{n-1}{3}.
\end{equation}
Summing over the $n$ choices of $r$ and applying
Proposition~\ref{prop:unweighting} gives
\[
 \score(T_1,T_2)\,=\,\frac4{n-4}W(T_1,T_2)
 \,\geq\,-\frac{4n}{n-4}
   \left(E+2\binom{n-1}{3}\right),
\]
so the factor $(n-4)/4$ in the unweighting identity has cancelled the
extra factor of $n$ from summing over the roots.  Using
\eqref{eq:score-distance},
\[
 \qd(T_1,T_2)
 \,\leq\,\frac23\binom n4+
   \frac{4n}{3(n-4)}\left(E+2\binom{n-1}{3}\right).
\]
Finally, $4n/(n-4)\,\leq\,8$ for $n\,\geq\,8$ and
$\binom{n-1}{3}\,\leq\,n^3$, which gives the stated bound.
\end{proof}

\begin{lemma}[uniform random relabelling]\label{lem:random-relabeling}
Let $n\,\geq\,4$, let $T$ be a binary phylogenetic tree on $X$, and let
$\pi$ be a uniformly random permutation of $X$.  Write $T^\pi$ for the
tree obtained by relabelling $T$ by $\pi$.  For every fixed
$Q\in\binom X4$, the three possible topologies of $T^\pi|_Q$ each have
probability $1/3$.  Consequently, for every pair $T_1,T_2$ on $X$,
there is a permutation $\pi$ such that
\[
 \qd(T_1,T_2^\pi)\,\geq\,\frac23\binom n4.
\]
\end{lemma}

\begin{proof}
Condition on the four-element set $S\,=\,\pi^{-1}(Q)$.  The quartet
$T|_S$ has a fixed split, while the restriction $\pi|_S$ is uniformly
distributed over the $4!$ bijections from $S$ to $Q$.  For each of the
three quartet topologies on $Q$, exactly eight of these bijections map the
two sides of the split of $T|_S$ to that topology: there are two ways to
exchange the sides and two ways to order the labels within each side.
Thus each topology has conditional, and hence unconditional, probability
$8/24\,=\,1/3$.

For fixed $T_1,T_2$, the expected disagreement on each quartet is
therefore $2/3$.  Linearity of expectation gives expected total distance
$(2/3)\binom n4$, and some permutation attains at least its expectation.
\end{proof}

\begin{proof}[Proof of Theorem~\ref{thm:main}]
Let $E_n\,=\,\varepsilon_n n^4$, where $\varepsilon_n\,\to\,0$ is supplied by
Theorem~\ref{thm:caterpillar}.  Then $E_n\,=\,o(n^4)$ and every pair of
$n$-leaf ordered caterpillars has score at least $-E_n$.  Applying
Proposition~\ref{prop:common-root-reduction} gives, uniformly over all
pairs of trees,
\[
 \qd(T_1,T_2)
 \,\leq\,\frac23\binom n4+\frac83(E_n+2n^3)
 \,=\,\left(\frac23+o(1)\right)\binom n4.
\]
Taking the maximum proves the upper bound in Theorem~\ref{thm:main}.

For the reverse inequality, apply
Lemma~\ref{lem:random-relabeling} to any pair of tree shapes.  Some
relabelling has distance at least $(2/3)\binom n4$.  This proves the
theorem.
\end{proof}

\section{Discussion}

Although the proof of Theorem~\ref{thm:main} as presented above invokes the
Alon--Naves--Sudakov caterpillar theorem and therefore ultimately rests on
flag algebras, the caterpillar result has an alternative proof from the
statistical literature on rank-based measures of association and consistent
tests of independence.  In that language, agreement of two caterpillars is
the four-point concordance event underlying the Bergsma--Dassios sign
covariance $\tau^*$ \cite{Bergsma_2014}.  After a discretization step, the
required lower bound follows from the nonnegativity of this coefficient.

This statistical argument brings together three measures of dependence.
Hoeffding introduced an integrated squared-discrepancy measure in his
nonparametric test of independence \cite{Hoeffding_1948}, and Blum, Kiefer,
and Rosenblatt introduced a closely related measure with respect to the
product of the marginals \cite{Blum_1961}.  Yanagimoto connected these
measures to probabilities of four-point rank patterns
\cite{Yanagimoto_1970}.  Bergsma and Dassios later introduced $\tau^*$
directly through concordant and discordant quadruples \cite{Bergsma_2014},
and Drton, Han, and Shi made explicit the resulting relation among the three
quantities \cite{Drton_2020}; see also the comparison of rank correlations by
Shi, Drton, and Han \cite{Shi_2022}.  The relation identifies the excess
four-point concordance probability above $1/3$ with $\tau^*$, and expresses
$\tau^*$ as a nonnegative linear combination of the two squared-discrepancy
measures.

For the precise formulation, let $(Z,V)$ be a random
point in $[0,1]^2$ whose two marginals are uniform, and write
\[
 F(z,v)\,=\,\mathbb P(Z\,\leq\,z,\ V\,\leq\,v).
\]
Let $(Z_i,V_i)$, $i\,=\,1,\ldots,4$, be independent copies of $(Z,V)$.  The continuous
marginals ensure that ties occur with probability zero.  For real numbers
$z_1,z_2,z_3,z_4$, let
\[
 a(z_1,z_2,z_3,z_4)
 \,=\,\operatorname{sgn}\bigl(
 |z_1-z_2|+|z_3-z_4|-|z_1-z_3|-|z_2-z_4|
 \bigr),
\]
where $\operatorname{sgn}(0)\,=\,0$, and define
\[
 \tau^*\,=\,
 \mathbb E\bigl[a(Z_1,Z_2,Z_3,Z_4)a(V_1,V_2,V_3,V_4)\bigr].
\]
This is the unnormalized convention, whose maximum is $2/3$; multiplying it
by $3/2$ gives the unit-maximum normalization used by some authors.
Let $\mathcal G$ be the event that the partition into the two smaller and
two larger $Z$-coordinates is also the partition into the two smaller and
two larger $V$-coordinates.  Writing $ij\mid k\ell$ for the corresponding
partition of the four indices, $a$ takes the values $-1$, $1$, and $0$ on
$12\mid34$, $13\mid24$, and $14\mid23$, respectively.  If
$p\,=\,\mathbb P(\mathcal G)$, symmetry under permutations of the four sample
indices assigns probability $p/3$ to each of the three equal pairs of
partitions and probability $(1-p)/6$ to each of the six unequal ordered
pairs.  Consequently,
\[
 \tau^*\,=\,\frac{2p}{3}-\frac{1-p}{3}
 \,=\,p-\frac13.
\]

Define the unnormalized squared-discrepancy functionals
\[
 \Delta(F)
 \,=\,\int_{[0,1]^2}\bigl(F(z,v)-zv\bigr)^2\,dF(z,v)
\]
and
\[
 R(F)
 \,=\,\int_0^1\int_0^1
       \bigl(F(z,v)-zv\bigr)^2\,dz\,dv.
\]
Here $\Delta(F)$ is the unnormalized integral underlying Hoeffding's measure
of dependence \cite{Hoeffding_1948}; Hoeffding's $D$ is commonly normalized
as $30\Delta(F)$.  The functional $R(F)$ is the corresponding
squared-discrepancy integral of Blum, Kiefer, and Rosenblatt
\cite{Blum_1961}.  The identity discussed above now gives
\begin{equation}\label{eq:four-point-square-identity}
 \mathbb P(\mathcal G)-\frac13
 \,=\,\tau^*\,=\,12\Delta(F)+24R(F)\,\geq\,0.
\end{equation}
In particular,
\begin{equation}\label{eq:four-point-one-third}
 \mathbb P(\mathcal G)\,\geq\,\frac13.
\end{equation}

\begin{proposition}[elementary caterpillar bound]
\label{prop:elementary-caterpillar}
Every pair $C_1,C_2$ of ordered caterpillars on the same $n$-element leaf set,
with $n\,\geq\,4$, satisfies
\[
 \qd(C_1,C_2)
 \,\leq\,\left(\frac23+\frac6n\right)\binom n4
\]
and
\[
 \score(C_1,C_2)
 \,\geq\,-\frac{18}{n}\binom n4
 \,=\,-\frac34(n-1)(n-2)(n-3)
 \,\geq\,-\frac34n^3.
\]
The score bound is equivalent to the distance bound and is used in
Proposition~\ref{prop:common-root-reduction}.
\end{proposition}

\begin{proof}
Label the leaves in the order of $C_1$ as $1,\ldots,n$, and let $\pi(i)$ be
the position of leaf $i$ in the order of $C_2$; thus
$\pi\in S_n$.  For $i\,<\,j\,<\,k\,<\,\ell$, the two caterpillars agree
on this quartet precisely when
\[
 \max(\pi(i),\pi(j))\,<\,\min(\pi(k),\pi(\ell))
 \quad\text{or}\quad
 \max(\pi(k),\pi(\ell))\,<\,\min(\pi(i),\pi(j)).
\]
Call such a four-element subset good.

For a discrete application of \eqref{eq:four-point-one-third}, choose
$L$ uniformly from $\{1,\ldots,n\}$, let $U_1,U_2$ be independent uniform
random variables on $[0,1]$, independent also of $L$, and set
\[
 Z\,=\,\frac{L-1+U_1}{n},
 \qquad
 V\,=\,\frac{\pi(L)-1+U_2}{n}.
\]
Both marginals are uniform on $[0,1]$, the second because $\pi$ is a
permutation.  Take four independent copies and let $\mathcal L$ be the event
that their four values of $L$ are distinct.  Conditional on $\mathcal L$, the
set of those four values is uniform over
$\binom{\{1,\ldots,n\}}4$; their order by $Z$ is their order in $C_1$, while
their order by $V$ is their order in $C_2$.  Hence
\[
 \mathbb P(\mathcal G\mid\mathcal L)
 \,=\,\frac{A(\pi)}{\binom n4},
\]
where $A(\pi)$ is the number of good four-element subsets.  A union bound gives
\[
 \mathbb P(\mathcal L^{\mathsf c})
 \,\leq\,\binom42\frac1n\,=\,\frac6n.
\]
Together with \eqref{eq:four-point-one-third}, this implies
\[
 \frac13
 \,\leq\,\mathbb P(\mathcal G)
 \,\leq\,\mathbb P(\mathcal G\mid\mathcal L)
       +\mathbb P(\mathcal L^{\mathsf c})
 \,\leq\,\frac{A(\pi)}{\binom n4}+\frac6n.
\]
Therefore
\[
 A(\pi)\,\geq\,\left(\frac13-\frac6n\right)\binom n4,
\]
which gives the asserted distance bound because
$\qd(C_1,C_2)\,=\,\binom n4-A(\pi)$.  Finally, with
$N\,=\,\binom n4$,
\[
 \score(C_1,C_2)
 \,=\,3A(\pi)-N
 \,\geq\,-\frac{18}{n}\binom n4
 \,\geq\,-\frac34n^3.
\]
\end{proof}

Thus Theorem~\ref{thm:caterpillar} can be replaced by
Proposition~\ref{prop:elementary-caterpillar}, removing flag algebras from the
caterpillar result and strengthening its $o(n^4)$ score error to an explicit
$O(n^3)$ bound.  Retaining the exact score bound
\[
 E\,=\,\frac{18}{n}\binom n4
 \,=\,\frac34(n-1)(n-2)(n-3)
\]
in Proposition~\ref{prop:common-root-reduction} gives
\begin{align*}
 \qd(T_1,T_2)
 &\,\leq\,\frac23\binom n4+
 \frac{4n}{3(n-4)}
 \left(\frac34(n-1)(n-2)(n-3)+2\binom{n-1}{3}\right)\\
 &\,=\,\frac23\binom n4+
 \frac{13n(n-1)(n-2)(n-3)}{9(n-4)}\\
 &\,\leq\,\frac23\binom n4+\frac{13}{9}n^3
 \qquad(n\,\geq\,8).
\end{align*}
The last inequality uses
$(n-1)(n-2)(n-3)\,\leq\,n^2(n-4)$, which holds for $n\,\geq\,5$.
Together with Lemma~\ref{lem:random-relabeling}, this gives
\[
 \frac23\binom n4
 \,\leq\,M_n
 \,\leq\,\frac23\binom n4+\frac{13}{9}n^3
 \qquad(n\,\geq\,8).
\]

\begingroup
\small
\setlength{\bibsep}{2pt}
\bibliographystyle{unsrtnat}
\bibliography{references}
\endgroup

\end{document}